\documentclass[11pt]{amsart}
\usepackage{amsopn, amssymb, amscd, amsmath, amsthm}
\usepackage{hyperref}
\usepackage{graphicx, graphics, epsfig}
\usepackage{enumerate}

\usepackage{todonotes} 

\usepackage{lscape, pdflscape}  
\usepackage{array, multirow} 
\usepackage{afterpage, capt-of} 

\allowdisplaybreaks

\newcommand{\nc}{\newcommand}

\nc{\Uca}{\mathcal{U}}

\nc{\RS}{\sigma_{\operatorname{Ric}}} \nc{\REV}{\operatorname{RicEV}}

\nc{\pr}{\operatorname{pr}}

\nc{\Gv}{{\G_{(\vg_i)} } }   \nc{\ggov}{{\ggo_{(\vg_i)} }}   \nc{\Gvt}{{\G^t_{(\vg_i)} } } 
\nc{\ig}{\mathfrak{i}}

\nc{\Gl}{\mathsf{GL}} \nc{\Or}{\mathsf{O}}  \nc{\SO}{\mathsf{SO}}   \nc{\Sl}{\mathsf{SL}}
\nc{\G}{\mathsf{G}} \nc{\K}{\mathsf{K}}  \nc{\T}{\mathsf{T}} \nc{\Lsf}{\mathsf{L}}
\nc{\Qb}{\mathsf{Q}_\Beta} \nc{\Hb}{\mathsf{H}_\Beta} \nc{\Ub}{\mathsf{U}_\Beta}
\nc{\Gb}{\mathsf{G}_\Beta} \nc{\Kb}{\mathsf{K}_\Beta}
\nc{\PPP}{\mathsf{P}} \nc{\U}{\mathsf{U}} \nc{\N}{\mathsf{N}} \nc{\Ss}{\mathsf{S}} \nc{\Aa}{\mathsf{A}}
\nc{\Hh}{\mathsf{H}}

\nc{\la}{\langle} \nc{\ra}{\rangle}

\nc{\laH}{\la\!\la} \nc{\raH}{\ra\!\ra}

\nc{\ipH}{{\laH \cdot, \cdot \raH}}

\nc{\iph}{{\la h \cdot , h \cdot \ra}}

\nc{\Vg}{{V(\ggo)}}

\nc{\alert}{\color{blue}}

\nc{\fg}{\mathfrak{f}}  \nc{\vg}{\mathfrak{v}} \nc{\wg}{\mathfrak{w}} \nc{\zg}{\mathfrak{z}} \nc{\ngo}{\mathfrak{n}} \nc{\kg}{\mathfrak{k}} \nc{\mg}{\mathfrak{m}} \nc{\bg}{\mathfrak{b}} \nc{\ggo}{\mathfrak{g}} \nc{\ggob}{\overline{\mathfrak{g}}} \nc{\sog}{\mathfrak{so}} \nc{\sug}{\mathfrak{su}} \nc{\spg}{\mathfrak{sp}} \nc{\slg}{\mathfrak{sl}} \nc{\glg}{\mathfrak{gl}} \nc{\cg}{\mathfrak{c}} \nc{\rg}{\mathfrak{r}}  \nc{\hg}{\mathfrak{h}} \nc{\tgo}{\mathfrak{t}} \nc{\ug}{\mathfrak{u}} \nc{\dg}{\mathfrak{d}} \nc{\ag}{\mathfrak{a}} \nc{\pg}{\mathfrak{p}} \nc{\sg}{\mathfrak{s}} \nc{\affg}{\mathfrak{aff}} \nc{\qg}{\mathfrak{q}}
\nc{\Xg}{\mathfrak{X}} \nc{\lgo}{\mathfrak{l}}

\nc{\pca}{\mathcal{P}} \nc{\nca}{\mathcal{N}} \nc{\lca}{\mathcal{L}} \nc{\oca}{\mathcal{O}} \nc{\mca}{\mathcal{M}} \nc{\tca}{\mathcal{T}} \nc{\aca}{\mathcal{A}} \nc{\cca}{\mathcal{C}} \nc{\gca}{\mathcal{G}} \nc{\sca}{\mathcal{S}} \nc{\hca}{\mathcal{H}} \nc{\bca}{\mathcal{B}} \nc{\dca}{\mathcal{D}}

\nc{\vp}{\varphi} \nc{\ddt}{\tfrac{{\rm d}}{{\rm d}t}} \nc{\dds}{\tfrac{{\rm d}}{{\rm d}s}} \nc{\ddtbig}{\frac{{\rm d}}{{\rm d}t}} \nc{\dd}{{\rm d}}
\nc{\dpar}{\tfrac{\partial}{\partial t}} \nc{\im}{\mathtt{i}}

 \nc{\SU}{\mathsf{SU}} 

\nc{\RR}{{\mathbb R}} \nc{\HH}{{\mathbb H}} \nc{\CC}{{\mathbb C}} \nc{\ZZ}{{\mathbb Z}}
\nc{\FF}{{\mathbb F}} \nc{\NN}{{\mathbb N}} \nc{\QQ}{{\mathbb Q}} \nc{\PP}{{\mathbb P}}

\nc{\vs}{\vspace{.2cm}} \nc{\vsp}{\vspace{1cm}} \nc{\ip}{{\langle\cdot,\cdot\rangle}}
\nc{\ipp}{(\cdot,\cdot)} \nc{\unm}{\tfrac{1}{2}}
\nc{\unc}{\tfrac{1}{4}} \nc{\und}{\tfrac{1}{16}} \nc{\no}{\vs\noindent}
\nc{\lam}{\Lambda^2(\RR^n)^*\otimes\RR^n} \nc{\tangz}{{\rm T}^{\rm Zar}}
\nc{\lamg}{\Lambda^2\ggo^*\otimes\ggo}
\nc{\nor}{{\sf n}}  \nc{\mum}{/\!\!/} \nc{\kir}{/\!\!/\!\!/}
\nc{\Ri}{\tfrac{4\Ric_{\mu}}{||\mu||^2}} \nc{\ds}{\displaystyle}
\nc{\lb}{[\cdot,\cdot]} \nc{\isn}{\tfrac{1}{||v||^2}}
\nc{\gkp}{(\ggo=\kg\oplus\pg,\ip)} \nc{\ukh}{(\ug=\kg\oplus\hg,\ip)}
\nc{\tgkp}{(\tilde{\ggo}=\kg\oplus\pg,\ip)}
\nc{\wt}{\widetilde}
\nc{\raw}{\rightarrow} \nc{\lraw}{\longrightarrow} \nc{\hqn}{\mathcal{H}_{q,n}}

\nc{\Spec}{\operatorname{Spec}}

\nc{\ad}{\operatorname{ad}}  \nc{\Aut}{\operatorname{Aut}}   \nc{\Inn}{\operatorname{Inn}}   \nc{\Lie}{\operatorname{Lie}} \nc{\Ad}{\operatorname{Ad}} \nc{\Der}{\operatorname{Der}} \nc{\rad}{\operatorname{rad}} \nc{\kf}{\operatorname{B}}

\nc{\End}{\operatorname{End}} \nc{\rank}{\operatorname{rank}} \nc{\Ker}{\operatorname{Ker}} \nc{\tr}{\operatorname{tr}} \nc{\sym}{\operatorname{sym}} \nc{\diag}{\operatorname{diag}} \nc{\proy}{\operatorname{pr}} \nc{\Adj}{\operatorname{Adj}} \nc{\vspan}{\operatorname{span}}

\nc{\Hess}{\operatorname{Hess}}  \nc{\dif}{\operatorname{d}} \nc{\sen}{\operatorname{sen}} \nc{\grad}{\operatorname{grad}} \nc{\Order}{\operatorname{O}} \nc{\divg}{\operatorname{div}}

\nc{\Iso}{\operatorname{Iso}} \nc{\Diff}{\operatorname{Diff}} \nc{\Rc}{\operatorname{Rc}} \nc{\Ricci}{\operatorname{Ric}} \nc{\Riem}{\operatorname{Rm}} \nc{\scalar}{\operatorname{sc}} \nc{\scalarm}{\hat{\operatorname{R}}} \nc{\Riccim}{\widehat{\operatorname{Ric}}} \nc{\tang}{\operatorname{T}} \nc{\vol}{\operatorname{vol}}

\nc{\mm}{\operatorname{M}} \nc{\CH}{\operatorname{CH}} \nc{\Irr}{\operatorname{Irr}} \nc{\mcc}{\operatorname{mcc}} \nc{\m}{\operatorname{m}}

\nc{\Id}{\operatorname{Id}}  \nc{\mmm}{\operatorname{m}}

\theoremstyle{plain}
\newtheorem{theorem}{Theorem}[section]
\newtheorem{proposition}[theorem]{Proposition}
\newtheorem{corollary}[theorem]{Corollary}
\newtheorem{lemma}[theorem]{Lemma}
\newtheorem{teointro}{Theorem}
\newtheorem{corointro}[teointro]{Corollary}

\theoremstyle{definition}

\theoremstyle{plain}
\newtheorem*{conjecture*}{Conjecture}
\newtheorem{conjecture}{Conjecture}

\theoremstyle{remark}
\newtheorem{remark}[theorem]{Remark}

\newtheorem{example}[theorem]{Example}

\title[On the signature of the Ricci curvature on nilmanifolds]{On the signature of the Ricci curvature on nilmanifolds}

\author{Romina M.~Arroyo}
\address{FaMAF $\&$ CIEM, Universidad Nacional de C\'ordoba, Av. Medina Allende s/n, Ciudad Universitaria, CP:X5000HUA, C\'ordoba, Argentina}
\email{arroyo@famaf.unc.edu.ar}

\author{Ramiro A.~Lafuente}
\address{School of Mathematics and Physics, The University of Queensland, St Lucia QLD 4072, Australia}
\email{r.lafuente@uq.edu.au}

\thanks{The first named author was partially supported by grants from the Australian Government through the Australian Research Council's Discovery Projects funding scheme (project DP180102185), CONICET, FONCYT and SeCyT (Universidad Nacional C\'ordoba). The second named author is an ARC DECRA fellow.
 }

\begin{document}

\begin{abstract}
We completely describe the signatures of the Ricci curvature of left-invariant Riemannian metrics on  arbitrary real nilpotent Lie groups. The main idea in the proof is to exploit a link between the kernel of the Ricci endomorphism and closed orbits in a certain representation of the general linear group, which we prove using the `real GIT' framework for the Ricci curvature of nilmanifolds. 
\end{abstract}

\maketitle


\section{Introduction}

A classical problem in Riemannian geometry is to determine the possible signatures of the Ricci curvature on a given space. For instance, the Bonnet-Myers theorem \cite{Mye41} states that a complete Riemannian manifold with $\Ricci \geq c > 0$ is compact and has finite fundamental group. On the other hand, any smooth manifold of dimension $n\geq 3$ admits a complete metric with $\Ricci<0$ \cite{Loh94}. 

In this article, we consider the problem under symmetry assumptions. More precisely, given a homogeneous space $M^n = \G/\Hh$, 
we are interested in the set
\[
	\RS(\G / \Hh) := \left\{ \sigma(\Ricci(g))  :  g \hbox{ $\G$-invariant Riemannian metric on }\G/\Hh \right\},
\]
where $\sigma(\Ricci(g)) = (s^-,s^0, s^+) \in \ZZ_{\geq 0}^3$ denotes the signature of the symmetric $(0,2)$-tensor  $\Ricci(g)$. Even under the homogeneity assumption, a complete description of this set turns out to be elusive in most cases. Partial results include Bochner's theorem ($(n,0,0)\notin \RS(\G/\Hh)$ for compact $\G$), and the classification of Lie groups $\G$ admitting a left-invariant metric with $\Ricci \geq 0$ \cite{BB}. 
The only semisimple Lie groups (up to covering) for which $\RS(\G)$ is known are $\Sl_2(\RR)$ and $\SU(2)$ \cite{Mln}. In fact, it is unknown whether $(6,0,0) \in \RS(\Sl_2(\CC))$.
 On the other hand, the recent literature suggests that determining whether $(n,0,0) \in \RS(\G)$ for solvable $\G$  could be out of reach, see \cite{LW19} and the references therein. 

We focus on the case where $\G = \N$ is a connected nilpotent Lie group, with $\Hh$  trivial by effectivenes.
Our main result is a complete description of $\RS(\N)$ for all such $\N$, in terms of purely Lie-theoretic data:

\begin{teointro}\label{thm_main}
The set of signatures of the Ricci curvature of left-invariant metrics on a connected nilpotent Lie group $\N$ with Lie algebra $(\ngo,[\cdot,\cdot])$ is given by 
\begin{align*}
	\RS(\N) \, &=  \, \bigcup_{r=0}^{\min(a_\ngo,m_\ngo)}  \left\{ (s^-,s^0,s^+) \, :  \,  
	 s^-  \geq u_\ngo{+}r ,  \,\, s^0 \geq a_\ngo{-}r, \,\, s^+ \geq z_\ngo {+} r , \, \,  s^- {+} s^0 {+} s^+ = \dim \ngo \right\}. 
\end{align*}
Here, $\zg(\ngo)$ denotes the center of $\ngo$, and  $u_\ngo, a_\ngo, z_\ngo, m_\ngo \in \ZZ_{\geq 0}$ are defined by
\begin{align*}
	 u_\ngo :=& \dim \ngo- \dim \left([\ngo,\ngo] + \zg(\ngo) \right), \qquad
	 & a_\ngo := \dim \zg(\ngo) - \dim \left(\zg(\ngo)\cap [\ngo,\ngo] \right), \\
 	z_\ngo  :=& \dim \left(\zg(\ngo)\cap [\ngo,\ngo] \right), \qquad
 	& m_\ngo := \dim [\ngo,\ngo] - \dim \left(\zg(\ngo)\cap [\ngo,\ngo] \right).
\end{align*}
\end{teointro}

In particular, we have the following

\begin{corointro}\label{cor_main}
If $\N$ is nilpotent and its  Lie algebra $\ngo$ satisfies $\zg(\ngo) \subset [\ngo,\ngo]$, then 
\[
	\RS(\N)  =  \left\{ (s^-,s^0,s^+)  \, :\,  s^- \geq u_\ngo, \,\, s^0\geq 0, \,\, s^+ \geq z_\ngo, \,\, s^- {+} s^0 {+} s^+ = \dim \ngo \right\}.
\]
\end{corointro}

This applies for example when $\ngo$ is irreducible, i.e.~not a direct sum of proper ideals. Indeed, any nilpotent Lie algebra  is a sum of ideals $\ngo = \RR^a \oplus \ngo_1$ with $\zg(\ngo_1) \subset [\ngo_1, \ngo_1]$ and $\RR^a$ abelian. Recall that 
the Ricci curvature of a left-invariant metric  $g$ on a nilpotent Lie group satisfies 
\begin{equation*}
	 \Ricci(X,X) = -\unm \sum_{i,j} g\left( [X,X_i], X_j \right)^2  + \unc \sum_{i,j}  g \left( [X_i, X_j], X \right)^2, 
\end{equation*}
see e.g.~\cite[(7.33)]{Bss}. Here $X\in \ngo$ is a left-invariant vector field and $\{X_i\}$ is a left-invariant $g$-orthonormal frame. It immediately follows that $\Ricci |_{\zg(\ngo)} > 0$, $\Ricci|_{[\ngo,\ngo]^\perp} < 0$, again assuming $\zg (\ngo) \subset [\ngo,\ngo]$ (cf.~ \cite{CHN17}). Thus, $s^+ \geq  z_\ngo$ and $s^- \geq u_\ngo$, for all $(s^-,s^0,s^+) \in \RS(\N)$.  Corollary \ref{cor_main} states that, besides these obvious restrictions, the Ricci curvature can take arbitrary signatures.


The signatures of the Ricci curvature on non-compact connected Lie groups $\G$ have been extensively investigated in the literature. In this case, we sometimes denote it by $\RS(\ggo)$, where $\ggo$ is the Lie algebra of $\G$. Wolf proved in \cite{Wolf69} that solvable $\ggo$ do not admit non-flat left-invariant metrics with $\Ricci \geq 0$. If $\ngo$ is nilpotent not abelian, then every $(s^-,s^0,s^+) \in \RS(\ngo)$ satisfies $s^-, s^+ \geq 1$  \cite{Mln} and $s^- \geq 2$ \cite{ChNik12}.  A complete characterisation of $\RS(\sg)$ for $\sg$ unimodular and two-step solvable was obtained by Dotti \cite{Dotti82}. In \cite{DBW16} the authors determined $\RS(\ngo)$ for nilpotent Lie algebras admitting a \emph{nice basis} and satisfying $m=1$ (in the notation of Theorem \ref{thm_main}) and $\zg(\ngo) \subset [\ngo,\ngo]$. They also computed $\RS(\ngo)$ for $\ngo$ nilpotent of dimension up to $6$, although we believe their classification contains some mistakes, see Section \ref{sec_counter}. Other previous results in low dimensions include Milnor's article solving the $3$-dimensional case \cite{Mln}, work by Kremlev and Nikonorov dealing with the problem in dimension $4$ \cite{KrmNkn1,KrmNkn2}, and Kremlev's resolution of the problem for nilpotent Lie algebras of dimension $5$ \cite{Krm}. The particular case of Ricci negative left-invariant metrics has attracted substantial attention in recent years, see for instance \cite{JblPet14,DttLtMtl,NN15,DL19,LW19,Will17,Will19}.

The proof of Theorem \ref{thm_main} in the case  $\zg(\ngo) \subset [\ngo,\ngo]$ has two main ingredients. Firstly, by an `Implicit Function Theorem'-kind of argument it suffices to show that $(u,m,z) \in \RS(\ngo)$. Secondly, in order to show that there exist a metric whose Ricci curvature has an $m$-dimensional radical, we apply a result relating closed orbits in the $\Gl(\ngo)$-representation space $\Lambda^2(\ngo^*) \otimes \ngo$, to zeroes of  Ricci (Proposition \ref{prop_closedzeroes}). The later relies on the `moment map' interpretation of the Ricci curvature on nilpotent Lie groups \cite{minimal}, and is inspired by the Kempf-Ness theorem \cite{KN79} (see also \cite{RS90} for the real version). Remarkably, it can also be applied to non-reductive subgroups of $\Gl(\ngo)$, and this is crucial in the proof.

In the general case $\ngo = \RR^a \oplus \ngo_1$, $\zg(\ngo_1) \subset [\ngo_1, \ngo_1]$, the proof goes essentially along the same lines. However, it  is sligthly more technical as one needs to keep track of the `angle' between the two summands $\RR^a$ and $\ngo_1$, which need not be orthogonal. This is captured by the variable $r$ in the statement.

The article is organised as follows. In Section \ref{sec_Ric} we review some basic facts about the Ricci curvature of nilmanifolds, and prove the easier inclusion in Theorem \ref{thm_main}. Section \ref{sec_counter} contains a $5$-dimensional counterexample to a conjecture stated in \cite{DBW16}. In Section \ref{sec_closedorbits} we show that certain orbits are closed in the representation space $\Lambda^2(\ngo^*)\otimes \ngo$. We then apply these results in Section \ref{sec_Riczeroes} to produce zeroes of the Ricci curvature. Finally, after establishing key properties of the linearisation of the Ricci curvature map in Section \ref{sec_linRic}, we prove our main result in Section \ref{sec_proof}.

\vs \noindent {\it Acknowledgements.} The authors would like to thank Christoph B\"ohm for useful comments on a first draft of this article. Part of this research was carried out while the first named author was a Postdoctoral Research Fellow at The University of Queensland. She is very grateful to the staff and students of the School of Mathematics \& Physics for their kindness and hospitality.

\section{The Ricci curvature of nilmanifolds}\label{sec_Ric}

In this section we review some well-known formulae for the Ricci curvature of left-invariant Riemannian metrics on nilpotent Lie groups. 

Let $\N$ be a connected real nilpotent Lie group with Lie algebra $\ngo$. Given a left-invariant metric $g$ on $\N$, its Ricci curvature tensor $\Ricci_g \in S^2(T^* \N)$ is also left-invariant. Hence both tensors are determined by their value at the identity $e\in \N$:
\[
	g(e) =: \ip  \in S^2(\ngo^*) , \qquad \Ricci_g(e) =: \Ricci_{\ip}  \in S^2(\ngo^*).
\]
It is well known (see \cite{Dotti82,Bss,soliton}) that for nilpotent $\ngo$ the Ricci curvature is given by
\begin{align}
	 \Ricci_\ip(X,Y) 
	 =& -\unm \sum_{i,j} \la \mu(X,X_i), X_j \ra \, \la  \mu(Y,X_i), X_j\ra \label{eqn_formulaRic} \\
	 & + \unc \sum_{i,j} \la \mu(X_i, X_j), X\ra \, \la \mu(X_i,X_j), Y \ra , \qquad X,Y\in \ngo, \nonumber
\end{align}
where $\{X_i \}$ denotes an arbitrary $\ip$-orthonormal basis for $\ngo$ and $\mu \in \Lambda^2(\ngo^*) \otimes \ngo$ denotes the Lie bracket. In coordinates, using the structure coefficients $\mu(X_i,X_j) = \mu_{ij}^k X_k $  (summation convention over repeated indices being used), one has
\begin{equation}\label{eqn_Ric_mu}
	\Ricci_\ip(X_r,X_s) = -\unm  \, \mu_{ri}^j \mu_{si}^j + \unc  \, \mu_{ij}^r \mu_{ij}^s =: \la \Ricci_\mu X_r, X_s \ra.
\end{equation}
In other words, the above implicitly defines $\Ricci_\mu$, an endomorphism of $\ngo$ whose matrix representation in the basis $\{X_i \}$ is also the matrix representation of the bilinear form $\Ricci_\ip$ in that basis.

Notice that when changing the metric $\ip$ for another  scalar product $\iph$, $h\in \Gl(\ngo)$, we may take as orthonormal basis $\{ h^{-1}X_i\}$, and the corresponding Ricci curvature will satisfy 
\[
	 \Ricci_\iph (h^{-1} X_r, h^{-1} X_s) = -\unm  \, (h\cdot \mu)_{ri}^j (h\cdot \mu)_{si}^j + \unc  \, (h\cdot \mu)_{ij}^r (h\cdot \mu)_{ij}^s =: \left\la \Ricci_{h\cdot \mu} X_r, X_s \right\ra ,
\]
The coefficients $(h\cdot \mu)_{ij}^k$ are of course the structure coefficients of $\mu$ with respect to the basis $\{h^{-1}X_i\}$. These coincide with the structure constants of $h\cdot \mu$  with respect to the \emph{original} basis $\{ X_i \}$, $(h \cdot \mu)(X_i, X_j) = (h\cdot \mu)_{ij}^k X_k$. Here, $h\cdot \mu$ denotes the standard 'change of basis' action of $\Gl(\ngo)$ on $\Lambda^2(\ngo^*)\otimes \ngo$, given by
\begin{equation}\label{eqn_action}
	(h \cdot \mu)(\cdot,\cdot) := h \mu(h^{-1} \cdot, h^{-1} \cdot), \qquad h\in \Gl(\ngo), \quad \mu \in \Lambda^2(\ngo^*)\otimes \ngo.
\end{equation}

Observe that we may write the Ricci curvature as a difference
\begin{equation}\label{eqn_Ricpq}
	\Ricci_\ip =  -  q_\ip +  p_\ip,
\end{equation}
where 
\begin{equation}\label{eqn_pq}
\begin{aligned}
	 q_\ip (X, Y) &=  \unm \la \ad_\mu X, \ad_\mu Y\ra,  \\
	 p_\ip (X, Y) &= \unc  \sum_{i,j} \la \mu(X_i,X_j), X\ra\la \mu(X_i,X_j), Y\ra, \qquad X,Y\in \ngo, 
\end{aligned}
\end{equation}
are positive semi-definite bilinear forms, with radicals $\rad (p_\ip) = \mu(\ngo,\ngo)^\perp$, $\rad (q_\ip) = \zg(\ngo, \mu)$. (Recall that the radical of a symmetric bilinear form $b(\cdot, \cdot)\in S^2(W^*)$ is the set $\rad(b) := \{ v \in W : b(v,w) = 0, \,\, \forall w\in W \}$;   for a semi-definite form one has $\rad(b) = \{ w\in W : b(w,w) = 0 \}$.) Equation \eqref{eqn_Ricpq} immediately yields

\begin{lemma}\label{lem_kerRic}
Any scalar product $\ip$ on a nilpotent Lie algebra $(\ngo,\mu)$ satisfies
\[
	\zg(\ngo,\mu) \cap \mu(\ngo,\ngo)^\perp \subset \rad \Ricci_\ip.
\]
\end{lemma}

The following lemma allows us to compute $\sigma(\Ricci_\ip)$ in terms of $\sigma(p_\ip)$ and $\sigma(q_\ip)$. Recall that the signature $\sigma(b)$ of a symmetric bilinear form $b(\cdot,\cdot)$ on a vector space $V$ is the unique triple $(s^-, s^0, s^+) \in \ZZ_{\geq0}^3$ with  $s^- + s^0 + s^+ = \dim V$, and $s^\pm$ the maximal dimension of a subspace where $\pm b$ is positive definite.

\begin{lemma}\label{lem_signaturest}
Let $s(\cdot,\cdot), t(\cdot,\cdot) \in S^2(W^*)$ be two positive semi-definite symmetric bilinear forms on a finite-dimensional real vector space $W$, with  $ \sigma(s) = (0,b,c)$, $\sigma(t) = (0,c,b)$ and $\rad(s) \cap \rad(t) = 0$. Then, $\sigma(t-s) = (c,0,b)$.  
\end{lemma}
\begin{proof}
Using $\dim \rad(s) + \dim \rad(t) = b + c = \dim W$ and $\rad(s) \cap \rad(t) = 0$ yields $W = \rad(s) \oplus \rad(t)$. This in turn implies that $(t-s)|_{\rad(s) \times \rad(s)} = t|_{\rad(s) \times \rad(s)} > 0$, $(t-s)|_{\rad(t) \times \rad(t)} = - s|_{\rad(t) \times \rad(t)} < 0$, and the lemma follows.
\end{proof}

\begin{lemma}\label{lem_easydirection}
Let $\ip$ be a scalar product on a nilpotent Lie algebra $(\ngo,\mu)$, and set 
\[
	r_{\mu} := \dim \zg(\ngo,\mu) - \dim \big( \zg(\ngo,\mu) \cap \mu(\ngo,\ngo)^\perp  \big) - \dim \big( \zg(\ngo,\mu) \cap \mu(\ngo,\ngo) \big).
\]
Then, $(s^-, s^0, s^+) := \sigma(\Ricci_\ip)$ satisfies 
\[
	s^- \geq u+r_\mu, \qquad s^0 \geq a-r_\mu, \qquad s^+ \geq z + r_\mu.
\]
In particular, the inclusion $\subseteq$ in Theorem \ref{thm_main} holds.
\end{lemma}

\begin{proof}
Set $\bg := \zg(\ngo,\mu) \cap \mu(\ngo,\ngo)^\perp$, $\zg_1 := \zg(\ngo,\mu) \cap \mu(\ngo,\ngo)$, $r_\mu := a - \dim \bg \geq 0$. Consider an orthogonal decomposition 
\begin{equation}\label{eqn_dec_n}
	\ngo =  \overbrace{\bg\oplus \cg}^{\mu(\ngo,\ngo)^\perp} \oplus \overbrace{\mg \oplus \zg_1}^{\mu(\ngo,\ngo)}.
\end{equation}

By Lemma \ref{lem_kerRic}, $s^0 = \dim \rad \Ricci_\ip \geq \dim \bg = a-r_\mu$.
Regarding  $s^-$, we observe that $\Ricci_\ip|_{\cg \times \cg}$ is negative definite. Indeed, $\cg \subset \rad(p_\ip)$ and $\cg \cap \rad(q_\ip) = 0$, thus $\Ricci_\ip(X,X) = -q_\ip(X,X) < 0$ for $X\in \cg\backslash \{0\}$.
Hence, 
\begin{align*}
		s^- \geq \dim \cg  &= \,\,  \dim \ngo - \dim \mu(\ngo,\ngo) - \dim \bg  \\
			&= \dim \ngo - \dim(\mu(\ngo,\ngo) + \zg(\ngo,\mu)) + \dim(\zg(\ngo,\mu)) - \dim(\zg \cap \mu(\ngo,\ngo))  - \dim \bg  \\
			& = u + a - (a-r_\mu) = u + r_\mu.
\end{align*}

It remains to be shown that $s^+ \geq z + r_\mu$. To that end, we will show that there exist two $r_\mu-$dimensional subspaces $W_1 \subset \mu(\ngo,\ngo)^\perp$ and $W_2\subset \mg$ such that the restriction of $\Ricci_\ip$ to $W_1\oplus W_2 \oplus \zg_1$ has signature $(r_\mu,0,z+r_\mu)$. 

We have the decomposition \eqref{eqn_dec_n} with $\bg \oplus \zg_1 \subset \zg(\ngo,\mu)$. Thus, $\ad(\cg + \mg) = \ad \ngo \simeq \ngo/\zg(\ngo,\mu)$. On the other hand, $\ad|_\cg$ and $\ad|_\mg$  are both injective, since $(\cg \oplus \mg) \cap \zg(\ngo,\mu) = 0$. Therefore,
\begin{align*}
	\dim  \left( \ad \cg  \cap \ad \mg\right) =& \,\,  \dim \ad \cg + \dim \ad \mg  -  \dim \ad(\cg + \mg)\\
		=&  \,\,  (u+r_\mu) + m - \dim \ngo/\zg(\ngo,\mu) \\
		=&  \,\, u+r_\mu + m - n + z + a = r_\mu.
\end{align*}
This means that there are $r_\mu$-dimensional subspaces $W_1 \subset \cg$, $W_2 \subset \mg$  such that $\ad W_1 = \ad W_2$.

Set now $W:= W_1\oplus W_2 \oplus \zg_1$, $s := q_\ip|_{W\times W}$, $t:= p_\ip|_{W\times W}$. We first observe that  $\rad(s) = \zg(\ngo,\mu) \cap W$ and $\rad(t) = \mu(\ngo,\ngo)^\perp \cap W$, thus $\rad(s) \cap \rad(t)  = \bg \cap W = 0$. Also, since $W_1\subset \mu(\ngo,\ngo)^\perp$ and $W_2\oplus \zg_1 \subset \mu(\ngo,\ngo)$, we have that  $\sigma(t) = (0,r_\mu,z+r_\mu)$. On the other hand,
\begin{align*}
	\dim \rad (s) &= \,\, \dim W \cap \zg(\ngo,\mu)  = \dim \ker (\ad |_W) = \dim W - \dim \ad(W) \\
		 &= \,\, 2r_\mu+z - r_\mu = z+r_\mu,
\end{align*}
since $\ad W = \ad W_1$ is $r_\mu$-dimensional.
Thus,  $\sigma(s) = (0,z+r_\mu, r_\mu)$. We may now apply Lemma \ref{lem_signaturest} to conclude that the signature of $\Ricci_\ip |_{W\times W} = t - s$ is $(r_\mu,0, z+r_\mu)$, as desired.
\end{proof}


\section{A counterexample}\label{sec_counter}

It has recently been conjectured that the set $\sigma \Ricci(\ngo)$ can be described as follows:

\begin{conjecture}\cite{DBW16}\label{conj_old}
The set of all possible signatures for the Ricci curvature of left-invariant Riemannian metrics on a nilpotent Lie group with Lie algebra $\ngo$ can be described in terms of the constants given in Theorem \ref{thm_main} as
\begin{align*}
	\RS(\ngo) \, &=  \, \bigcup_{r=0}^{\min(a,m)}  \left\{ (s^-,s^0,s^+) \, :  \,  
	 s^-  \geq u{+}r ,  \,\, s^0 \geq a{-}r, \,\, s^+ \geq z , \, \,  s^- {+} s^0 {+} s^+ = \dim \ngo \right\}. 
\end{align*}
\end{conjecture}

It is not hard to see that this conjectural set contains the one stated in Theorem \ref{thm_main}, and that the inclusion is strict unless $\zg(\ngo,\mu) \subset \mu(\ngo,\ngo)$ (if the latter happens then $a=0$, thus also $r=0$). Using Lemma \ref{lem_easydirection} we can quickly state an explicit counterexample:

\begin{example}
Consider the $5$-dimensional nilpotent Lie algebra $\ngo$ with basis $\{X_i\}_{i=1}^5$ and non-zero Lie brackets given by
\[
	\mu(X_1, X_2)= X_3, \qquad \mu(X_1,X_3) = X_4. 
\]
(This Lie algebra is denoted by $L_{5,3}$ in  \cite{DBW16}.) It satisfies $n=5$, $z = a = m = 1$, $u=2$. Thus, according to Conjecture \ref{conj_old} we should have $(4,0,1)\in \RS(\ngo)$ (setting $r=1$). However, from Lemma \ref{lem_easydirection} it follows that  $(s^-, s^0, s^+) \in \RS(\ngo)$ implies $s^0 + s^+ \geq a + z = 2$, a contradiction.
\end{example}

\section{Subgroups of $\Gl(\ngo)$ whose orbits are closed}\label{sec_closedorbits}

In this section we will produce closed subgroups of $\Gl(\ngo)$ whose orbits through $\mu \in \Lambda^2(\ngo^*)\otimes\ngo$ are closed. Let us first set up some notation. Given a  Lie algebra $\ngo$, a fixed `background' scalar product $\ip$ induces scalar products (also denoted by $\ip$) on $\ngo^*$ and on any tensor product. For example, given any orthonormal basis $\{ X_i\}$ of $(\ngo,\ip)$ with dual basis $\{X^i \}$, the induced scalar products  on $\End(\ngo) \simeq \ngo^* \otimes \ngo$  and $\Lambda^2(\ngo^*) \otimes \ngo$ have orthonormal bases given by $\{ X^i \otimes X_j \}$ and $\{ (X^i \wedge X^j) \otimes X_k \}$, respectively. Of course, the one on $\End(\ngo)$ may be alternatively defined by $\la A , B \ra := \tr A B^t$, $A,B\in \End(\ngo)$, where the transpose is defined with respect to $\ip$.

 Let $\left(\ngo^{(i)} \right)_{i\geq 0}$ denote the descending central series of a Lie algebra $(\ngo,\mu)$: 
\[
	\ngo^{(0)} = \ngo, \qquad \ngo^{(i+1)} := \mu\big(\ngo, \ngo^{(i)}\big) \quad \hbox{ for }i\geq 0.
\]
By definition, $(\ngo,\mu)$ is nilpotent if and only if $\ngo^{(N)} = 0$ for some $N \in \NN$. We will assume this is the case from now on.

Given any direct sum decomposition $\ngo = \vg_1 \oplus \vg_2  \oplus \vg_3$ into subspaces $(\vg_i)_{i=1}^3$, consider the following subset of $\Gl(\ngo)$:
\begin{equation}\label{eqn_Gv}
	\Gv :=  \left\{ h\in  \Gl(\ngo)  : h \big|_{\vg_1} = \Id_{\vg_1}, \, h \big|_{\vg_2} = \Id_{\vg_2}, \, h(\vg_3) \subset \vg_2 \oplus \vg_3  \right\}.
\end{equation}
It is clear that $\Gv$ is a closed Lie subgroup of $\Gl(\ngo)$, with Lie algebra
\[
	\ggov = \left\{  A\in \End(\ngo) : A \big|_{\vg_1} =  A\big|_{\vg_2} = 0, \, \, A (\vg_3) \subset \vg_2 \oplus \vg_3 \right\}.
\]

\begin{remark}
The reader interested in understanding the proof of Theorem \ref{thm_main} in the case $\zg(\ngo,\mu) \subset \mu(\ngo,\ngo)$ may assume that $\vg_1 = 0$ throughout this and the following sections.
\end{remark}

The next lemma is one of the main ingredients in the proof of Theorem \ref{thm_main}:
\begin{lemma}\label{lem_closedorb}
Let $\ngo = \vg_1 \oplus \vg_2  \oplus \vg_3$ be a nilpotent Lie algebra with Lie bracket $\mu \in \Lambda^2 \ngo^* \otimes \ngo$. Assume that 
$(\vg_1 \oplus \vg_2) + \mu(\ngo,\ngo) = \ngo$ and $\zg(\ngo,\mu) \subset \vg_1 \oplus \vg_2$. Then, the orbit $\Gv\cdot \mu $ is closed.
\end{lemma}

Its proof requires the following fact about nilpotent Lie algebras.

\begin{lemma}\label{lem_subalgebra}
Let $\hg$ be a subalgebra of a nilpotent Lie algebra $\ngo$ with $\ngo = \hg + \mu(\ngo,\ngo)$. Then, $\hg = \ngo$.
\end{lemma}

\begin{proof}
By nilpotency it is enough to prove that $\ngo = \hg + \ngo^{(r)}$ for all $r\geq 1$. We establish this by induction, the case $r=1$ being the lemma assumption. If  $\ngo = \hg + \ngo^{(r)}$, then
\begin{align*}
	\ngo &= \,\, \hg + \mu(\ngo,\ngo) = \hg + \mu(\hg +  \ngo^{(r)}, \hg +  \ngo^{(r)}) \\
		&  \subset \,\,  \hg + \mu(\hg,\hg) + \mu(\hg,  \ngo^{(r)}) +  \mu(\ngo^{(r)},  \ngo^{(r)})
		\subset \hg + \ngo^{(r+1)},
\end{align*}
and the claim follows.
\end{proof}


\begin{proof}[Proof of Lemma \ref{lem_closedorb}]
Consider a sequence $(h^{(k)})_{k\geq 1} \subset \Gv$ such that   $\lim_{k\to \infty} h^{(k)} \cdot \mu =: \bar \mu$ exists. We first claim that for each $v\in \ngo$, the sequence $(h^{(k)}v)_{k\in \NN}$ is bounded. To see that, set
\[
	\hg := \left\{ v\in \ngo : (h^{(k)} v)_{k\in \NN} \hbox{ is bounded} \right\} \subset \ngo.
\]
It is clearly a vector subspace of $\ngo$. Moreover, given $v,w \in \hg$, we have 
\begin{equation}\label{eqn_hkmu}
	h^{(k)} \left(\mu(v,w) \right) = (h^{(k)} \cdot \mu) \left(h^{(k)} v, h^{(k)} w\right), 
\end{equation}
which is bounded uniformly in $k$ since $\big(h^{(k)} \cdot \mu\big)_{k\in \NN}$ is bounded in $\Lambda^2 (\ngo^*) \otimes \ngo$. Thus, $\hg$ is a Lie subalgebra of $(\ngo,\mu)$. Also, $\vg_1 \oplus \vg_2 \subset \hg$ by definition of $\Gv$, thus by assumption we must have that $\hg + \mu(\ngo,\ngo) = \ngo$. Lemma \ref{lem_subalgebra} now yields $\hg = \ngo$.

The above claim implies that,  after passing to a subsequence (which by simplicity we denote with the same indices), $h^{(k)}$ converges to some linear map $\bar h$. It is enough to show that $\bar h$ is invertible. Indeed, this would imply that $\bar \mu = \lim h^{(k)} \cdot \mu =  \bar h \cdot \mu$, as desired. 

Assume on the contrary this is not the case and let 
\[
	\ig := \{ v\in \ngo  : h^{(k)} v \to_{k\to\infty} 0\} \neq 0.
\]
By \eqref{eqn_hkmu}, $\ig$ is an ideal in $(\ngo,\mu)$, and $\ig \cap \zg(\ngo,\mu) = 0$ since by assumption $\zg(\ngo,\mu) \subset \vg_1 \oplus \vg_2$. This contradicts the fact that any nonzero ideal of a nilpotent Lie algebra must intersect its center \cite[p.13]{Hum78}.
\end{proof}

\section{Closed orbits and zeroes of the Ricci curvature}\label{sec_Riczeroes}

We now recall one of the most remarkable and useful facts about the Ricci curvature of nilmanifolds:  formula \eqref{eqn_Ricmu} below, relating it to the 'GIT moment map' of the $\Gl(\ngo)$-representation \eqref{eqn_action}. Its origins may be traced back to \cite[$\S$6.4]{Heber1998}. Due to our needs in the present article, and in order to simplify the presentation, we have decided to  avoid discussing real GIT, referring instead the interested reader to \cite{RS90,HS07, EbJbl09,realGIT} and the references therein.

It was observed in \cite[Prop.~3.5]{minimal} that in terms of the Lie bracket $\mu\in \Lambda^2(\ngo^*)\otimes \ngo $ of $\ngo$, the Ricci curvature of a nilmanifold satisfies 
\begin{equation}\label{eqn_Ricmu}
 	\la \Ricci_\mu , A \ra = \unc \la  \pi(A) \mu, \mu\ra = \tfrac18 \,  \ddt\big|_0 \, \left\Vert \exp(t A) \cdot \mu \right\Vert^2, \qquad  	A\in \End(\ngo),
\end{equation}
where $\pi : \End(\ngo) \to \End(\Lambda^2(\ngo^*)\otimes \ngo)$ is the Lie algebra representation determined by differentiating the Lie group action \eqref{eqn_action} at the identity, and explicitly given by $\left(\pi(A)\mu\right) (\cdot,\cdot) := A \mu(\cdot, \cdot) - \mu(A \cdot, \cdot) - \mu(\cdot, A \cdot)$. 


The next result is a simple generalisation of one of the directions in the Kempf-Ness theorem (see \cite{KN79}, and \cite{RS90} for the $\RR$-version). We point out that neither the representation space $\Lambda^2(\ngo^*)\otimes \ngo$ nor the fact that $\ngo$ is nilpotent are essential here: one simply replaces the Ricci curvature by the real GIT moment map, and obtains a similar result for  closed orbits in appropriate representations spaces of real reductive Lie groups (cf.~\cite{realGIT}).


\begin{proposition}\label{prop_closedzeroes}
Let $\mu \in \Lambda^2(\ngo^*)\otimes \ngo$ be the Lie bracket of a nilpotent Lie algebra and $\G \subset \Gl(\ngo)$ a Lie subgroup with Lie algebra $\ggo$. Assume that the orbit $\G \cdot \mu$ is closed. Then, there exists a bracket $\bar \mu \in \G \cdot \mu$ whose Ricci curvature satisfies $\Ricci_{\bar \mu} \perp A$ for all $A\in \ggo$.
\end{proposition}
\begin{proof}
For the closed set $\G\cdot \mu$ there exists $\bar\mu \in \G \cdot \mu$ minimising the distance to the origin. 
In particular, $\bar \mu$ is a critical point for $\Vert \cdot \Vert^2 \big|_{\G\cdot \mu}$, thus \eqref{eqn_Ricmu} gives $\la \Ricci_{\bar\mu}, A \ra = 0$ for all $A \in \ggo$.
\end{proof}

The following  consequence of Lemma \ref{lem_closedorb} and Proposition \ref{prop_closedzeroes} yields the existence of a left-invariant metric whose Ricci endomorphism has a `large' kernel:


\begin{corollary}\label{cor_Riczero} 
Let $\ngo = \vg_1 \oplus \vg_2 \oplus \vg_3$ be an orthogonal decomposition with 
\[
	\vg_1 = \zg(\ngo,\mu) \cap \mu(\ngo,\ngo)^\perp, \qquad (\vg_1 \oplus \vg_2) + \mu(\ngo,\ngo) = \ngo, \qquad   \zg(\ngo,\mu) \subset \vg_1 \oplus \vg_2.
\]  
Then, there exists  $\bar \mu \in \Gv \cdot \mu$ such that  $\vg_1 \oplus \vg_3 \subset \ker \Ricci_{\bar\mu}$.
\end{corollary}


\begin{proof}
Lemma \ref{lem_closedorb} implies that the orbit $\Gv \cdot \mu$ is closed, and from Proposition \ref{prop_closedzeroes} applied to the Lie subgroup $\Gv$ we deduce the existence of $\bar \mu \in \Gv \cdot \mu$ such that $\Ricci_{\bar \mu} \perp \ggov$. Let us show that this $\bar \mu$ satisfies the  above stated property.

Firstly, since the elements of $\Gv$ act trivially on $\zg(\ngo,\mu) \subset \vg_1 \oplus \vg_2$, and they preserve $\vg_2\oplus \vg_3$, we have that $\vg_1 \subset \zg(\ngo,\bar \mu) \cap \bar\mu(\ngo,\ngo)^\perp$. Therefore, Lemma \ref{lem_kerRic} implies  that $\vg_1 \subset \ker \left(\Ricci_{\bar \mu}\right)$.
To conclude the proof, let us see that $\vg_3 \subset \ker \Ricci_{\bar \mu}$. Let $X\in \vg_3$, $\Vert X \Vert = 1$, and consider   $A\in \End(\ngo)$ with $A X = \Ricci_{\bar \mu} X$, $A Y = 0$ for all $Y\perp X$. Since $\vg_1 \subset \ker \Ricci_{\bar \mu}$, we must have $\Ricci_{\bar \mu} X \perp \vg_1$, from which $A\in \ggov$ and therefore  $0 = \tr A \Ricci_{\bar \mu} = \Vert \Ricci_{\bar \mu} X \Vert^2$.
\end{proof}

\section{The linearisation of the Ricci curvature}\label{sec_linRic}

Let $\pg:= \{A\in \glg(\ngo) : A = A^T \}$. For a subspace $\sg \subset \pg$ we denote by $\pr_\sg : \pg \to \sg$ the corresponding orthogonal projection. 

 We view the Ricci endomorphism as a map 
\[
	\Ricci : \Gl(\ngo) \cdot \bar \mu \to \pg,  \qquad h\cdot  \bar \mu \mapsto \Ricci_{h\cdot \bar\mu}.
\]
Let $L_{\bar \mu} : \pg \to \pg$ be the linear map given by
\begin{equation}\label{eqn_defL}
		L_{\bar \mu}(E) = {\rm d} \Ricci|_{\bar \mu} (\pi(E)\bar\mu).
\end{equation}
Recall that $T_{\bar \mu} (\Gl(\ngo)\cdot \bar \mu) = \{ \pi(E)\bar\mu : E \in \glg(\ngo) \}$.

Notice that $L_{\bar \mu}$ is self-adjoint. Indeed, if $\bar\mu(t) := \exp( t E) \cdot \bar \mu$, then by \eqref{eqn_Ricmu} we have
\begin{align*}
	\la  L_{\bar \mu} E, F\ra &= \,\, \ddt\big|_{t=0} \la \Ricci_{\bar \mu(t)}, F  \ra 
		= \ddt\big|_{t=0} \unc \, \la \pi(F) \bar \mu(t), \bar \mu(t) \ra \\
		&= \,\, \unc \, \la \pi(F) \pi(E)\bar \mu, \bar \mu \ra + \unc \, \la \pi(F) \bar \mu, \pi(E)\bar \mu \ra
		= \unm \, \la \pi(E) \bar \mu, \pi(F) \bar \mu\ra,
\end{align*}
for any $E, F\in \pg$, where in the last equality we have used the fact that $\pi(F)^T = \pi(F^T) = \pi(F)$. This computation also shows that $L_{\bar \mu}$ is positive semi-definite, with 
\[
	\ker L_{\bar \mu} = \Der(\bar \mu) \cap \pg.
\]
Indeed, $\Der(\bar \mu) = \{ E\in \glg(\ngo) : \pi(E)\bar \mu  = 0\}$. In particular, we have

\begin{lemma}\label{lem_Lsurj}
The projection $\pr_\sg \circ L_{\bar \mu} : \pg \to \sg$ is surjective if and only if $\sg \cap \Der(\bar \mu) = 0$.
\end{lemma}

\begin{proof}
We will show that the orthogonal complement of $\pr_\sg \circ L_{\bar \mu} (\pg)$ in $\sg$ equals $\sg \cap \Der(\bar \mu)$. If $S \in \sg$ belongs to the former, then also $S\perp L_{\bar \mu}(\pg)$. Since $L_{\bar \mu}$ is self-adjoint, this implies that $S\in \ker L_{\bar \mu} = \Der(\bar \mu) \cap \pg$, so $S \in \sg \cap \Der(\bar \mu)$. Conversely, let $S\in \sg \cap \Der(\bar \mu)$. Then $S\in \ker L_{\bar \mu}$, thus for any $E\in \pg$ we have that
\[
		\la \pr_{\sg} \circ L_{\bar \mu}(E) , S \ra = \la L_{\bar \mu}(E), S \ra = \la E, L_{\bar \mu}(S) \ra = 0.
\]
\end{proof}

\section{Proof of Theorem \ref{thm_main}}\label{sec_proof}











By Lemma \ref{lem_easydirection} we know that one of the inclusions holds. Let us now show that  any triple as in the theorem's statement can be realised as the signature of the Ricci curvature of some left-invariant metric.

Any nilpotent Lie algebra may written as 
\[
	\ngo = \ag \oplus \ngo_1, \qquad  \zg(\ngo,\mu) = \ag \oplus \zg_1, \qquad \zg_1:= \zg(\ngo,\mu) \cap \mu(\ngo,\ngo).
\] 
The subspaces $\ag$, $\ngo_1$ are nilpotent ideals, $\ag$ is central, and  we have a Lie algebra direct sum $\ngo \simeq \RR^a \oplus \ngo_1$. Clearly, $\mu(\ngo,\ngo)\subset \ngo_1$.  Choose direct complements $\ug$ of $\mu(\ngo,\ngo)$ in $\ngo_1$, and $\mg$ of $\zg_1$ in $\mu(\ngo,\ngo)$, so that we have the decompositions
\[
	\ngo = \ag \oplus \overbrace{\ug \oplus \underbrace{ \mg \oplus \zg_1}_{\mu(\ngo,\ngo)}}^{\ngo_1}, \qquad \zg(\ngo,\mu) = \ag \oplus \zg_1.
\]
Let us fix an inner product $\ip$ on $\ngo$ making the above decompositions orthogonal.

We begin by making the following reduction. Choose an integer $r \in [0, \min(a,m)]$ and consider any orthogonal decomposition $\ag = \ag_0 \oplus \ag_1$ into subspaces, where $\dim \ag_0 = a-r$, $\dim \ag_1 = r$. Then $\ngo = \ag_0 \oplus (\ag_1 \oplus \ngo_1)$, with $\tilde \ngo := \ag_1 \oplus \ngo_1$ a nilpotent ideal. The  simply-connected  Lie group $\N$ with left-invariant metric $g$ corresponding to $(\ngo,\ip)$  decomposes as a Riemannian product $\N = \RR^{a-r} \times \tilde \N$, where the first factor is Euclidean (flat) and the second one is the simply-connected Lie group with Lie algebra $\tilde \ngo$, endowed with the corresponding left-invariant metric $\tilde g$. We clearly have 
\[
		\sigma(\Ricci(g)) = (0,a-r,0) + \sigma(\Ricci(\tilde g)).
\]
The theorem will follow if we show that for any triple of non-negative integers $(m^-, m^0, m^+)$ with $m^-+m^0 +m^+ = m-r$ we have 
\begin{equation}
	(u+r,0,z+r) + (m^-,m^0,m^+) \in \sigma \Ricci(\tilde \ngo).
\end{equation}
From now on and for the rest of the proof we will focus on proving this assertion. To ease notation,  we will simply ignore the subspace $\ag_0$ and assume that $\ngo = \tilde \ngo$. That is, we have
\[
	r = a := \dim \ag \leq \dim \mg
\] 
and aim to prove that for all $(m^-, m^0, m^+) \in  \ZZ_{\geq 0}^3$ with $m^-+m^0 +m^+ = m-a$ we have 
\begin{equation}\label{eqn_toshow}
	(u+a,0,z+a) + (m^-,m^0,m^+) \in \sigma \Ricci(\ngo).
\end{equation}

\vs

In terms of our fixed background scalar product $\ip$ on $\ngo$, we  may parametrise all scalar products on $\ngo$ via $\iph$, $h\in \Gl(\ngo)$. Recall that by \eqref{eqn_Ric_mu}, the endomorphism $\Ricci_{h\cdot \mu}$ represents the bilinear form $\Ricci_\iph$ in a certain basis. Thus,  Sylverster's law of intertia allows us to compute the signature of $\Ricci_\iph$ by looking at the signs of the eigenvalues of  $\Ricci_{h\cdot \mu}$. Throughout the proof, when refering to the signature of $\Ricci_{h\cdot \mu}$, which we will write simply as $\sigma\left(\Ricci_{h\cdot \mu} \right)$, we will always mean the signature of the bilinear form $\la \Ricci_{h\cdot \mu} \cdot, \cdot \ra$.

The first step towards proving \eqref{eqn_toshow} is to establish that $(u,a+m,z) \in \sigma\Ricci(\ngo)$. To  that end, we apply Corollary \ref{cor_Riczero} to the subspaces $\vg_1 = \ag$, $\vg_2 = \ug\oplus \zg_1$, $\vg_3 = \mg$. This yields a bracket $\bar \mu\in \Gv \cdot \mu$ with $\ag \oplus \mg \subset \ker \Ricci_{\bar \mu}$.
In other words, $s^0 \geq a+m$, where $(s^-, s^0, s^+) = \sigma(\Ricci_{\bar \mu})$. On the other hand, $s^- \geq u$, $s^+ \geq z$ by Lemma \ref{lem_easydirection}.
Since $a+m + u + z = \dim \ngo$ we conclude that in fact we have $\sigma(\Ricci_{\bar \mu}) = (u,a+m,z)$.

The strategy  is now to prove that local variations of $\bar\mu$ within the orbit $\Gl(\ngo)\cdot \bar\mu$ yield  all claimed Ricci signatures. Thus, we study  $\sigma(\Ricci_{h\cdot \bar \mu})$ for $h \in \Gl(\ngo)$ in a neighbourhood of the identity. More precisely, we only consider $h = \exp(E)$, for $E\in \pg$ in a neighbourhood of $0$. According to the decomposition $\ngo = \vg_1 \oplus \vg_2 \oplus \vg_3$ we have 
\[
		\Ricci_{\bar \mu} = 
		\left[ \begin{matrix}
			0 & 0 & 0\\
			0 & R & 0\\
			0 & 0 & 0
		\end{matrix}
		\right],
\]
with $R$ non-singular and $\sigma(R) = (u,0,z)$. Let us write
\[
	\Ricci_{\exp(E)\cdot \bar\mu} = \Ricci_{\bar \mu} + A(E) =  
	\left[
		\begin{matrix}
			A_{11}(E) & A_{12}(E) & A_{13}(E) \\
			A_{12}(E)^T & R + A_{22}(E) & A_{23}(E) \\
			A_{13}(E)^T & A_{23}(E)^T & A_{33}(E) 
		\end{matrix}
	\right],
\]
with  $A(0) = 0$, $A_{i,j}(E) : \vg_j \to \vg_i$, $A_{ii}(E)^T = A_{ii}(E)$, and $E\mapsto A(E)$ a smooth map whose differential is given by
\begin{equation}\label{eqn_dA}
	{\rm d} A|_0 (E) = L_{\bar \mu}(E) = {\rm d} \Ricci|_{\bar \mu} (\pi(E)\bar\mu), \qquad E\in \pg.
\end{equation}
For small enough $E$ the operator $R+A_{22}(E)$ is invertible, allowing us to change basis with the endomorphism
\[
	Q:= 	\left[ 
	\begin{matrix}
			\Id & 0 & 0 \\
			-(R+ A_{22}(E))^{-1} A_{12}(E)^T &\Id & - (R + A_{22}(E))^{-1} A_{23}(E)\\ 
			0 & 0 & \Id
	\end{matrix}
	\right]
\]
to obtain an operator whose signature is easier to compute:
\[
	Q^T \, \Ricci_{\exp(E) \cdot \bar \mu} \, Q = 
	\left[
		\begin{matrix}
			X_{11}(E) 	& 0 			& X_{13}(E) \\
			0			& R + A_{22}(E) & 0 \\
			X_{13}(E)^T & 0				& X_{33}(E) 
		\end{matrix}
	\right].
\]
Here, $X_{11}, X_{13}, X_{33}$ are defined by
\begin{align*}
	X_{11} &:= \,\,  A_{11} - A_{12} \, (R+A_{22})^{-1} \, A_{12}^T  \,  : \,  \vg_1 \to \vg_1 \, , \\
	X_{13} &:=  \,\, A_{13} - A_{12} \, (R+A_{22})^{-1} \, A_{23}\,  : \, \vg_3 \to \vg_1 \, , \\
	X_{33} &:= \,\, A_{33} - A_{23}^T \, (R+A_{22})^{-1} \, A_{23} \, : \, \vg_3 \to \vg_3.
\end{align*}
A straightforward computation shows that their first variations are 
\begin{equation}\label{eqn_dXij}
		({\rm d} X_{i,j}) \big|_0 (E) = ({\rm d} A_{i,j})\big|_0 (E), \qquad \forall \, i,j\in\{1,3\}, \qquad E\in \pg.
\end{equation}
In order to compute the signature of $Q^T \, \Ricci_{\exp(E) \cdot \bar \mu} \, Q $, we need to understand $\sigma(X(E))$, where $X(E): \vg_1 \oplus \vg_3 \to \vg_1 \oplus \vg_3$ is given by
\[
	X(E) = \left[
		\begin{matrix}
			X_{11}(E) 		& X_{13}(E) \\
			X_{13}(E)^T 	& X_{33}(E) 
		\end{matrix}
	\right].
\]

Let $\sg := \{ A\in \pg  : A|_{\vg_2} = 0, A(\vg_1) \subset \vg_3 \}$, and  notice that $X_{33}(E)$, $X_{13}(E) + X_{13}(E)^T \in \sg$.

\begin{lemma}\label{lem_sDer0}
We have that $\sg \cap \Der(\bar \mu) = 0$.
\end{lemma}

\begin{proof}
Let $\tilde D\in \sg \cap \Der(\bar \mu)$. Then $\tilde D|_{\vg_2} = 0$. Write $\bar \mu = h\cdot \mu$, $h \in \Gv$, so that $D:= h^{-1} \tilde D h \in \Der(\mu)$. By definition of derivation, $\ker D \subset \ngo$ is a Lie subalgebra. Therefore $\hg:= \ker D \cap \ngo_1$ is a subalgebra of the ideal $\ngo_1$. Since  $\Gv$ acts trivially on $\vg_2$ we also have that $\vg_2 \subset \hg$.  Now $\hg + \mu(\ngo_1,\ngo_1) = \ngo_1$, and by Lemma \ref{lem_subalgebra} we conclude that $\hg = \ngo_1$. This implies that $\ngo_1 = \vg_2 \oplus \vg_3 \subset \ker D$. Now $\Gv$ preserves $\ngo_1$, from which we deduce that also $\vg_2 \oplus \vg_3 \subset \ker \tilde D$. Since $\tilde D$ is self-adjoint,  the later gives $\tilde D(\vg_1) \subset \vg_1$. By definition of $\sg$, this yields $\tilde D = 0$ and concludes the proof.
\end{proof}

By Lemmas \ref{lem_Lsurj} and \ref{lem_sDer0}, the orthogonal projection of $L_{\bar \mu}$ onto $\sg$ is surjective. Using \eqref{eqn_dA}, \eqref{eqn_dXij} and the Implicit Function Theorem, this implies that for some neighbourhood $\Uca$ of $0$ in $\pg$, the images $X_{13}(\Uca)$ and $X_{33}(\Uca)$ contain $0$ as an interior point. In other words, they  attain any given value whose norm is sufficiently small. As we will see, this is enough for constructing $E \in \pg$ such that $\Ricci_{\exp(E)\cdot \bar\mu}$ has the desired signature.

Indeed, consider an arbitrary orthogonal decomposition $\mg = \mg_1 \oplus \mg_2$, with $\dim \mg_1 = a$. Let $(m^-, m^0, m^+) \in \ZZ_{\geq 0}^3$ with $m^- + m^0 + m^+ = m-a$, and choose a self-adjoint endomorphism $Y: \mg_2 \to \mg_2$ such that $\sigma(Y) = (m^-, m^0, m^+)$. Finally, choose a self-adjoint linear isomorphism $ \ag \simeq \mg_1$. Then, by the above reasoning, there exists $E\in \pg$ such that
\[
		X(E) = \left[
		\begin{matrix}
			X_{11}(E) 		& \Id  & 0 \\
			\Id & 0 & 0 \\
			0  & 0	& Y 
		\end{matrix}
	\right],
\]
with blocks according to the decomposition $\ag \oplus \mg_1 \oplus \mg_2$.  Clearly, we may assume $E$ is small enough so that $\sigma(R+A_{22}(E)) = \sigma(R)$. Finally, we have
\begin{align*}
	\sigma(\Ricci_{\exp(E)\cdot \bar \mu}) &=  \,\, \sigma(Q^T \Ricci_{\exp(E)\cdot \bar \mu} Q) =     \sigma(R) + \sigma(X(E))  \\
	 & = \,\, \sigma(R) + \sigma(Y) + \sigma \left(  \left[
		\begin{matrix}
			X_{11}(E) & \Id \\
			\Id & 0 
		\end{matrix} \right] \right) \\
	& = (u,0,z) + (m^-, m^0, m^+) + (a,0,a),
\end{align*}
and \eqref{eqn_toshow} follows. The last equality uses the fact that the signature remains invariant along a continuous path of invertible, self-adjoint operators:
\[
	\sigma \left(  \left[
		\begin{matrix}
			X_{11}(E) & \Id \\
			\Id & 0 
		\end{matrix} \right] \right) = 
		\sigma \left(  \left[
		\begin{matrix}
			t\cdot X_{11}(E) & \Id \\
			\Id & 0 
		\end{matrix} \right] \right) = 
		\sigma \left(  \left[
		\begin{matrix}
			0 & \Id \\
			\Id & 0 
		\end{matrix} \right] \right) = (a,0,a), \qquad t\in [0,1].
\] 
This concludes the proof of the theorem.

\bibliography{bib/ramlaf2}

\def\cprime{$'$} \renewcommand{\MR}[1]{} \newcommand{\noop}[1]{}
\providecommand{\bysame}{\leavevmode\hbox to3em{\hrulefill}\thinspace}
\providecommand{\MR}{\relax\ifhmode\unskip\space\fi MR }
\providecommand{\MRhref}[2]{%
  \href{http://www.ams.org/mathscinet-getitem?mr=#1}{#2}
}
\providecommand{\href}[2]{#2}
\begin{thebibliography}{DNBWK16}

\bibitem[BB78]{BB}
Lionel B{\'e}rard-Bergery, \emph{Sur la courbure des m\'etriques riemanniennes
  invariantes des groupes de {L}ie et des espaces homog\`enes}, Ann. Sci.
  \'Ecole Norm. Sup. (4) \textbf{11} (1978), no.~4, 543--576.

\bibitem[Bes87]{Bss}
Arthur~L. Besse, \emph{Einstein manifolds}, Ergebnisse der Mathematik und ihrer
  Grenzgebiete (3) [Results in Mathematics and Related Areas (3)], vol.~10,
  Springer-Verlag, Berlin, 1987.

\bibitem[BL20]{realGIT}
Christoph B{\"o}hm and Ramiro~A. Lafuente, \emph{Real geometric invariant
  theory}, Differential Geometry in the Large, London Mathematical Society
  Lecture Note Series, Cambridge University Press, Cambridge, 2020.

\bibitem[CHN17]{CHN17}
Grant Cairns, Ana Hini{\'c}Gali{\'c}, and Yuri Nikolayevsky, \emph{Curvature
  properties of metric nilpotent lie algebras which are independent of metric},
  Annals of Global Analysis and Geometry \textbf{51} (2017), no.~3, 305--325.

\bibitem[CN12]{ChNik12}
M.~S. Chebarykov and Yu.~G. Nikonorov, \emph{The {R}icci operator of completely
  solvable metric {L}ie groups}, Mat. Tr. \textbf{15} (2012), no.~2, 146--158.
  \MR{3074460}

\bibitem[DL19]{DL19}
Jonas Der{\'e} and Jorge Lauret, \emph{On {R}icci negative solvmanifolds and
  their nilradicals}, Mathematische Nachrichten (in press, 2019).

\bibitem[DLM84]{DttLtMtl}
I.~Dotti, M.~L. Leite, and R.~J. Miatello, \emph{Negative {R}icci curvature on
  complex simple {L}ie groups}, Geom. Dedicata \textbf{17} (1984), no.~2,
  207--218.

\bibitem[DM82]{Dotti82}
Isabel Dotti~Miatello, \emph{Ricci curvature of left invariant metrics on
  solvable unimodular {L}ie groups}, Math. Z. \textbf{180} (1982), no.~2,
  257--263.

\bibitem[DNBWK16]{DBW16}
M.~B. Djiadeu~Ngaha, M.~Boucetta, and J.~Wouafo~Kamga, \emph{The signature of
  the {R}icci curvature of left-invariant {R}iemannian metrics on nilpotent
  {L}ie groups}, Differential Geom. Appl. \textbf{47} (2016), 26--42.
  \MR{3504917}

\bibitem[EJ09]{EbJbl09}
Patrick Eberlein and Michael Jablonski, \emph{Closed orbits of semisimple group
  actions and the real {H}ilbert-{M}umford function}, New developments in {L}ie
  theory and geometry, Contemp. Math., vol. 491, Amer. Math. Soc., Providence,
  RI, 2009, pp.~283--321. \MR{2537062}

\bibitem[Heb98]{Heber1998}
Jens Heber, \emph{Noncompact homogeneous {E}instein spaces}, Invent. Math.
  \textbf{133} (1998), no.~2, 279--352.

\bibitem[HS07]{HS07}
Peter Heinzner and Henrik St{\"o}tzel, \emph{Semistable points with respect to
  real forms}, Math. Ann. \textbf{338} (2007), no.~1, 1--9. \MR{2295501
  (2008c:32030)}

\bibitem[Hum78]{Hum78}
James~E. Humphreys, \emph{Introduction to {L}ie algebras and representation
  theory}, Graduate Texts in Mathematics, vol.~9, Springer-Verlag, New
  York-Berlin, 1978, Second printing, revised. \MR{499562}

\bibitem[JP17]{JblPet14}
Michael Jablonski and Peter Petersen, \emph{A step towards the {A}lekseevskii
  conjecture}, Math. Ann. \textbf{368} (2017), no.~1-2, 197--212. \MR{3651571}

\bibitem[KN79]{KN79}
George Kempf and Linda Ness, \emph{The length of vectors in representation
  spaces}, Algebraic geometry ({P}roc. {S}ummer {M}eeting, {U}niv.
  {C}openhagen, {C}openhagen, 1978), Lecture Notes in Math., vol. 732,
  Springer, Berlin, 1979, pp.~233--243. \MR{555701}

\bibitem[KN09]{KrmNkn1}
A.~G. Kremlev and Yu.~G. Nikonorov, \emph{The signature of the {R}icci
  curvature of left-invariant {R}iemannian metrics on four-dimensional {L}ie
  groups. {T}he unimodular case [translation of mr2500127]}, Siberian Adv.
  Math. \textbf{19} (2009), no.~4, 245--267.

\bibitem[KN10]{KrmNkn2}
\bysame, \emph{The signature of the {R}icci curvature of left-invariant
  {R}iemannian metrics on four-dimensional {L}ie groups. {T}he nonunimodular
  case [translation of mr2569648]}, Siberian Adv. Math. \textbf{20} (2010),
  no.~1, 1--57.

\bibitem[Kre09]{Krm}
A.~G. Kremlev, \emph{The signature of the {R}icci curvature of left-invariant
  {R}iemannian metrics on five-dimensional nilpotent {L}ie groups}, Sib.
  \`Elektron. Mat. Izv. \textbf{6} (2009), 326--339.

\bibitem[Lau01]{soliton}
Jorge Lauret, \emph{Ricci soliton homogeneous nilmanifolds}, Math. Ann.
  \textbf{319} (2001), no.~4, 715--733.

\bibitem[Lau06]{minimal}
\bysame, \emph{A canonical compatible metric for geometric structures on
  nilmanifolds}, Ann. Global Anal. Geom. \textbf{30} (2006), no.~2, 107--138.

\bibitem[Loh94]{Loh94}
Joachim Lohkamp, \emph{Metrics of negative {R}icci curvature}, Ann. of Math.
  (2) \textbf{140} (1994), no.~3, 655--683. \MR{1307899}

\bibitem[LW19]{LW19}
Jorge Lauret and Cynthia~E Will, \emph{On {R}icci negative {L}ie groups}, arXiv
  preprint arXiv:1912.06204 (2019).

\bibitem[Mil76]{Mln}
John Milnor, \emph{Curvatures of left-invariant metrics on {L}ie groups}, Adv.
  Math. \textbf{21} (1976), 293--329.

\bibitem[Mye41]{Mye41}
S.~B. Myers, \emph{Riemannian manifolds with positive mean curvature}, Duke
  Math. J. \textbf{8} (1941), 401--404. \MR{4518}

\bibitem[NN15]{NN15}
Yuri Nikolayevsky and Yu~G Nikonorov, \emph{On solvable lie groups of negative
  ricci curvature}, Mathematische Zeitschrift \textbf{280} (2015), no.~1-2,
  1--16.

\bibitem[RS90]{RS90}
Roger~Wolcott Richardson and Peter Slodowy, \emph{Minimum vectors for real
  reductive algebraic groups}, J. London Math. Soc. (2) \textbf{42} (1990),
  no.~3, 409--429. \MR{1087217 (92a:14055)}

\bibitem[Wil17]{Will17}
Cynthia~E Will, \emph{Negative ricci curvature on some non-solvable lie
  groups}, Geometriae Dedicata \textbf{186} (2017), no.~1, 181--195.

\bibitem[Wil19]{Will19}
Cynthia Will, \emph{Negative {R}icci curvature on some non-solvable {L}ie
  groups ii}, Math. Z. (in press) (2019).

\bibitem[Wol69]{Wolf69}
Joseph~A. Wolf, \emph{A compatibility condition between invariant riemannian
  metrics and {L}evi-{W}hitehead decompositions on a coset space}, Transactions
  of the American Mathematical Society \textbf{139} (1969), 429--442.

\end{thebibliography}
\bibliographystyle{amsalpha}

%

\end{document}